\documentclass[12pt]{article}
\usepackage{amsmath}
\usepackage{amsthm}
\usepackage{amsfonts}
\usepackage{amssymb}
\usepackage{latexsym}

\long\def\comment #1\commentend{}

\newtheorem{thm}{Theorem}[section]
\newtheorem{cor}[thm]{Corollary}
\newtheorem{lem}[thm]{Lemma}
\newtheorem{claim}[thm]{Claim}
\newtheorem{prop}[thm]{Proposition}

\numberwithin{equation}{section}

\newcommand{\RR}{\mathbb{R}}
\newcommand{\R}{\mathbb{R}}
\newcommand{\EE}{\mathbb{E}}
\newcommand{\EL}{\mathcal{E}}

\newcommand{\RN}{\RR^n}
\newcommand{\RK}{{\RR^k}}
\newcommand{\BK}{{B_2^k}}

\newcommand{\SN}{S^{n-1}}
\newcommand{\SK}{{S^{k-1}}}

\newcommand{\halfn}{\lfloor \frac{n}{2} \rfloor}
\newcommand{\NN}{\mathbb{N}}

\newcommand{\sumn}[1]{{\sum_{#1=1}^n}}
\newcommand{\sumk}[1]{{\sum_{#1=1}^k}}

\newcommand{\norm}[1]{\lVert #1\rVert}
\newcommand{\pnorm}[2]{\lVert #1\rVert_{#2}}

\newcommand{\modulus}[1]{\lvert #1\rvert}

\newcommand{\remark}{\noindent \textbf{Remark: }}

\newcommand{\e}{\epsilon}

\begin{document}

\title{Euclidean sections of convex bodies\\
 \normalsize Series of lectures given in
Bedlewo,
Poland, July 6-12, 2008\\ and in Kent, Ohio, August
13-20, 2008 }
\author{Gideon Schechtman\footnote {supported in part by the Israel Science Foundation}}
\date{}

\maketitle

This is a somewhat expanded form of a four hours course given, with small variations,  first at
the educational workshop Probabilistic methods in Geometry,
Bedlewo, Poland, July 6-12, 2008 and a few weeks later at the
Summer school on Fourier analytic and probabilistic methods in
geometric functional analysis and convexity, Kent, Ohio, August
13-20, 2008.\\
The main part of these notes gives yet another exposition
of Dvoretzky's theorem on Euclidean sections of convex bodies with a
proof based on Milman's. This material is by now quite standard.
Towards the end of these notes we discuss issues related to fine
estimates in Dvoretzky's theorem and there there are some results that
didn't appear in print before. In particular there is an exposition of
an unpublished result of Figiel (Claim \ref{claim:figiel}) which gives an upper bound
on the possible dependence on $\e$ in Milman's theorem. We would like
to thank Tadek Figiel for allowing us to include it here. There is
also a better version of the proof of one of the results from \cite{sc2}
giving a lower bound on the dependence
on $\e$ in Dvoretzky's theorem. The improvement is in the statement and proof of Proposition \ref{prop:main} here which is a stronger version of the corresponding Corollary 1 in \cite{sc2}.

\medskip

\section {Lecture 1}
By a convex, symmetric body $K\subset \RN$ we shall refer to a
compact set with non-empty interior which is convex and symmetric
about the origin (i.e, $x\in K$ implies that $-x\in K$.

This series of lectures will revolve around the following theorem of Dvoretzky.
\begin{thm}\label{thm_dvoretzky_0} (A. Dvoretzky, 1960)
There is a function $k:(0,1)\times\NN\to\NN$ satisfying, for all $0<\varepsilon<1$, $k(\varepsilon,n)\to\infty$ as $n\to\infty$,
such that for every $0<\varepsilon<1$, every $n\in \NN$ and every convex symmetric body in
$K\subset \RN$ there exists a subspace $V\subseteq \RN$
satisfying:
\begin{enumerate}
\item $\dim V = k(\varepsilon,n)$.
\item $V\cap K$ is ``$\varepsilon$-euclidean", which means that there
exists $r>0$ such that:
\[
r\cdot V\cap B_2^n \subset V\cap K \subset (1+\varepsilon)r\cdot
V\cap B_2^n.
\]
\end{enumerate}
\end{thm}

The theorem was proved by Aryeh Dvoretzky \cite{dv},
answering a question of Grothendieck. The question of Grothendieck was asked in
\cite{gr} in relation with a paper of Dvoretzky and Rogers \cite{dr}. \cite{gr} gives another proof of the main application (the
existence, in any infinite dimensional Banach space, of an
unconditionally convergent series which is not absolutely convergent)
of the result of Dvoretzky and Rogers \cite{dr} a version of which is used bellow (Lemma \ref{lem:dvor_rog}).

The original proof of Dvoretzky is very involved. Several
simplified proofs were given in the beginning of the
70-s; one by Figiel \cite{fi}, one by Szankowski \cite{sz} and the earliest one, a
version of which we'll present here, by Milman \cite{mi}. This proof which turn out to be very influential is
based on the notion of {\bf Concentration of Measure}. Milman was
also the first to get the right estimate ($\log n$) of the
dimension  $k=k(\varepsilon,n)$ of the almost euclidean section as the function of the dimension $n$. The
dependence of $k$ on $\varepsilon$ is still wide open and we'll
discuss it in detail later in this survey. Milman's version of Dvoretzky's theorem is the following.

\begin{thm}\label{thm_dvoretzky_1} 
For every $\varepsilon>0$ there exists a constant $c=c(\varepsilon)>0$
such that for every $n\in \NN$ and every convex symmetric body in
$K\subset \RN$ there exists a subspace $V\subseteq \RN$
satisfying:
\begin{enumerate}
\item $\dim V = k$, where $k\ge c\cdot \log n$.
\item $V\cap K$ is $\varepsilon$-euclidean:
\[
r\cdot V\cap B_2^n \subset V\cap K \subset (1+\varepsilon)r\cdot
V\cap B_2^n.
\]
\end{enumerate}
\end{thm}

For example, the unit ball of $\ell_\infty^n$ - the $n$-dimensional
cube - is far from the Euclidean ball. Its easy to see, that the
ratio of radii of the bounding and the bounded ball is $\sqrt{n}$:
\[
B_2^n \subset B_\infty^n \subset \sqrt{n} B_2^n
\]
and $\sqrt{n}$ is the best constant. Yet, according to Theorem \ref{thm_dvoretzky_1}, we can find a subspace of $\RN$ of dimension proportional
to $\log n$ in which the ratio of bounding and bounded balls will
be $1+\varepsilon$.

There is a simple correspondence between symmetric convex sets in
$\RN$ and norms on $\RN$ Given by $\pnorm{x}{K}=\inf \{
\lambda>0\;:\; \frac{x}{\lambda} \in K\}$ The following is an
 equivalent formulation of
Theorem \ref{thm_dvoretzky_1} in terms of norms.

\begin{thm}\label{thm_dvoretzky_2}
For every $\varepsilon>0$ there exist a constant $c=c(\varepsilon)>0$
such that for every $n\in \NN$ and every norm $\pnorm{\cdot}{}$ in
$\RN$ \ $\ell_2^k$ $(1+\varepsilon)$-embeds in
$(\RN,\pnorm{\cdot}{})$ for some $k\ge c\cdot \log n$.
\end{thm}

By ``$X$ $C$-embed in $Y$" I mean: There exists a one to one
bounded operator $T:X\to Y$ with $\|T\|\|(T_{|TX})^{-1}\|\le C$.

Clearly, Theorem \ref{thm_dvoretzky_1} implies Theorem
\ref{thm_dvoretzky_2}. Also, Theorem \ref{thm_dvoretzky_2} clearly
implies a weaker version of Theorem \ref{thm_dvoretzky_1}, with
$B_2^n$ replaced by some ellipsoid (which by definition is an
invertible linear image of $B_2^n$). But, since any
$k$-dimensional ellipsoid easily seen to have a $k/2$-dimensional
section which is a multiple of the Euclidean ball, we see that
also Theorem \ref{thm_dvoretzky_2} implies Theorem
\ref{thm_dvoretzky_1}. This argument also shows that proving
Theorem \ref{thm_dvoretzky_1} for $K$ is equivalent to proving it
for some invertible linear image of $K$.

Before starting the actual proof of Theorem \ref{thm_dvoretzky_2} here is \textbf{A Very vague sketch of the proof:} Consider the
unit sphere of $\ell_2^n$, the surface of $B_2^n$, which we will
denote by $\SN=\{x\in \RN\;:\; \pnorm{x}{2}=1\}$. Let
$\pnorm{x}{}$ be some arbitrary norm in $\RN$. The first task will
be to show that there exists a ``large" set
$S_{\text{good}}\subset \SN$ satisfying $\forall x\in
S_{\text{good}}.\; \modulus{\pnorm{x}{} - M} < \varepsilon M$ where
$M$ is the average of $\pnorm{x}{}$ on $\SN$. Moreover, we shall
see that, dependeing on the Lipschitz constant of $\|\cdot\|$, the
set $S_{\text{good}}$ is ``almost all" the sphere in the measure
sense. This phenomenon is called \textsl{concentration of
measure}.

\noindent The next stage will be to pass from the ``large" set to
a large dimensional subspace of $\RN$ contained in it. Denote
$O(n)$ - the group of orthogonal transformations from $\RN$ into
itself. Choose some subspace $V_0$ of appropriate dimension $k$
and fix an $\varepsilon$-net $N$ on $V_0\cap \SN$. For some $x_0\in
N$,``almost all" transformations $U\in O(n)$ will send it into
some point in $S_{\text{good}}$. Moreover, if the ``almost all"
notion is good enough, we will be able to find a transformation
that sends all the points of the $\varepsilon$-net into
$S_{\text{good}}$. Now there is a standard approximation procedure
that will let us pass from the $\varepsilon$-net to all points in the
subspace.

In preparation for the actual proof denote by $\mu$ the normalized Haar measure on $\SN$ - the unique,
probability measure which is invariant under the group of
orthogonal transformations. The main tool will be the following
concentration of measure theorem of Paul Levy (for a proof see e.g. \cite{sc1}).

\begin{thm}\label{thm_concentration}(P. Levy)
Let $f:\SN \longrightarrow \RR$ be a Lipshitz function with a
constant $L$; i.e.,
\[
\forall x,y\in \SN\; \modulus{f(x)-f(y)}\le L\pnorm{x-y}{2}.
\]
Then,
\[
\mu\{x\in\SN\;:\; \modulus{f(x)-Ef} > \varepsilon\} \le
2e^{-\frac{\varepsilon^2n}{2L^2}}.
\]
\end{thm}

\noindent \textbf{Remark:} The theorem also holds with the
expectation of $f$ replaced by its median.

\medskip

Our next goal is to prove the following theorem of Milman which,
gives some lower bound on the dimension of almost Euclidean section
in each convex body. It will be the main tool in the proof of Theorem \ref{thm_dvoretzky_2}.\\

\begin{thm}\label{thm_pre_dvoretzky}(V. Milman)
For every $\varepsilon>0$ there exists a constant $c=c(\varepsilon)>0$
such that for every $n\in \NN$ and every norm $\pnorm{\cdot}{}$ in
$\RN$ there exists a subspace $V\subseteq \RN$ satisfying:
\begin{enumerate}
\item $\dim V = k$, where $k\ge c\cdot \bigg( \frac{E}{b}
\bigg)^2n$. \item For every $x\in V$:
\[
(1-\varepsilon)E\cdot \pnorm{x}{2}\le \pnorm{x}{} \le
(1+\varepsilon)E\cdot \pnorm{x}{2}.
\]
\end{enumerate}
Here $E=\int_{\SN} \pnorm{x}{} d\mu$ and $b$ is the smallest
constant satisfying $\pnorm{x}{} \le b\pnorm{x}{2}$.
\end{thm}

The definition of $b$ implies that the function $\|\cdot\|$ is
Lipschitz with constant $b$ on $S^{n-1}$. Applying Theorem
\ref{thm_concentration} we get a subset of $S^{n-1}$ of probability
very close to one ($\ge 1-2e^{-\varepsilon^2E^2n/2}$), assuming $E$
is not too small, on which
\begin{equation}\label{eq_equivalence}
(1-\varepsilon)E \le \pnorm{x}{} \le (1+\varepsilon)E .
\end{equation}
We need to replace this set of large measure with a set which is
large in the algebraic sense: A set of the form $V\cap S^{n-1}$
for a subspace $V$ of relatively high dimension. The way to
overcome this difficulty is to fix an $\varepsilon$-net in $V_0\cap
S^{n-1}$ (i.e., a finite set such that any other point in $V_0\cap
S^{n-1}$ is of distance at most $\varepsilon$ from one of the points
in this set) for some {\em fixed} subspace $V_0$ (of dimension $k$
to be decided upon later) and show that we can find an orthogonal
transformation $U$ such that $\|Ux\|$ satisfies equation
\ref{eq_equivalence} for each $x$ in the $\varepsilon$-net. A
successive approximation argument (the details of which can be
found, e.g., in \cite{ms}, as all other details which are not
explained here), then gives a similar inequality (maybe with
$2\varepsilon$ replacing $\varepsilon$) for all $x\in V_0\cap S^{n-1}$,
showing that $V=UV_0$ can serve as the needed subspace.

To find the required $U\in O(n)$ we need two simple facts. The
first is to notice that if we denote by $\nu$ the normalized Haar
measure on the orthogonal group $O(n)$, then, using the uniqueness
of the Haar measure on $S^{n-1}$, we get that, for each fixed
$x\in S^{n-1}$, the distribution of $Ux$, where $U$ is distributed
according to $\nu$, is $\mu$. It follows that, for each fixed
$x\in S^{n-1}$, with $\nu$-probability at least
$1-2e^{-\varepsilon^2E^2n/2}$,
\[
(1-\varepsilon)E\le \pnorm{Ux}{} \le (1+\varepsilon)E .
\]
Using a simple union bound we get that for any finite set $N\subset
S^{n-1}$, with $\nu$-probability $\ge
1-2|N|e^{-\varepsilon^2E^2n/2}$, $U$ satisfies
\[
(1-\varepsilon)E\le \pnorm{Ux}{} \le (1+\varepsilon)E
\]
for all $x\in N$ ($|N|$ denotes the cardinality of $N$).

\begin{lem}\label{lem_first}
For every $0 < \varepsilon < 1$ there exists an $\varepsilon$-net $N$ on
$S^{k-1}$ of cardinality $\le \bigg( \frac{3}{\varepsilon} \bigg)^k$.
\end{lem}

So as long as, $2\bigg( \frac{3}{\varepsilon} \bigg)^k
e^{-\varepsilon^2E^2n/2}<1$ we can find the required $U$. This
translates into: $k\ge c\frac{\varepsilon^2}{\log \frac
{3}{\varepsilon}}E^2n$ for some absolute $c>0$ as is needed in the
conclusion of Theorem \ref{thm_pre_dvoretzky}.

\medskip

\noindent \textbf{Remark:} This proof gives that the $c(\varepsilon)$ in Theorem  \ref{thm_pre_dvoretzky} can be taken to be $c\frac{\varepsilon^2}{\log \frac
{3}{\varepsilon}}$ for some absolute $c>0$. This can be improved to $c(\varepsilon)\ge c\varepsilon^2$ as was done first by Gordon in \cite{go}. (See also \cite{sc0}) for a proof that is more along the lines here.) This later estimate can't be improved as we shall see below in Claim \ref{claim:figiel}.

\medskip

To prove the lemma, let $N=\{x_i\}_{i=1}^m$ be a maximal set in
$S^{k-1}$ such that for all $x,y\in N$ $\pnorm{x-y}{2} \ge
\varepsilon$. The maximality of $N$ implies that it is an
$\varepsilon$-net for  $S^{k-1}$. Consider
$\{B(x_i,\frac{\varepsilon}{2})\}_{i=1}^m$ - the collection of balls
of radius $\frac{\varepsilon}{2}$ around the $x_i$-s. They are
mutually disjoint and completely contained in
$B(0,1+\frac{\varepsilon}{2})$. Hence:
\[
m Vol \bigg(B(x_1,\frac{\varepsilon}{2}) \bigg)=\sum Vol
\bigg(B(x_i,\frac{\varepsilon}{2}) \bigg)=Vol \bigg(\bigcup
B(x_i,\frac{\varepsilon}{2}) \bigg)  \le Vol
\bigg(B(0,1+\frac{\varepsilon}{2}) \bigg).
\]
The $k$ homogeneity of the Lebesgue measure in $\RR^k$ implies now
that $m \le \bigg(\frac{1+\varepsilon/2}{\varepsilon/2} \bigg)^k =
\bigg( 1+ \frac{2}{\varepsilon} \bigg)^k$.

This completes the sketch of the proof of Theorem
\ref{thm_pre_dvoretzky}. \qed

\section {Lecture 2}

In order to prove Theorem
\ref{thm_dvoretzky_2} we need to estimate $E$ and $b$ for a general symmetric convex body. Since the
problem is invariant under invertible linear transformation we may
assume that $S^{n-1}$ is included in $K$, i.e., $b=1$. In remains
to estimate $E$ from below. As we'll see this can be done quite
effectively for many interesting examples (we'll show the
computation for the $\ell_p^n$ balls). However in general it may
happen that $E$ is very small even if we assume as we may that
$S^{n-1}$ touches the boundary of $K$. This is easy to see.

The way to overcome this difficulty is to assume in addition that
$S^{n-1}$ is the ellipsoid of maximal volume inscribed in $K$. An
ellipsoid is just an invertible linear image of the canonical
Euclidean ball. Given a convex body one can find by compactness an
ellipsoid of maximal volume inscribed in it. It is known that
this maximum is attained for a unique inscribed ellipsoid but this fact will not be used in the reasoning below. The invariance of the problem lets us assume that
the canonical Euclidean ball is such an ellipsoid. The advantage
of this special situation comes from the following Lemma

\begin{lem}\label{lem:dvor_rog} (Dvoretzky-Rogers)
Let $\pnorm{\cdot}{}$ be some norm on $\RN$ and denote its unit
ball by $K=B_{\pnorm{\cdot}{}}$. Assume the Euclidean ball
$B_2^n=B_{\|\cdot\|_2}$ is (the) ellipsoid of maximal volume
inscribed in $K$. Then there exist and orthonormal basis $x_1,
\dots, x_n$ such that
\[
e^{-1} (1 - \frac{i-1}{n}) \le \pnorm{x_i}{} \le 1, \;\;\; \text{
for all } 1 \le i \le n.
\]
\end{lem}

\remark This is a weaker version of the original Dvoretzky-Rogers
lemma. It shows in particular that half of the $x_i$-s have norm
bounded from below: for all $1 \le i \le \halfn \;\; \norm{x_i}
\ge (2e)^{-1}$. This is what will be used in the proof of the main
theorem.

\begin{proof}
First of all choose an arbitrary $x_1 \in S^{n-1}$ of maximal
norm. Of course, $\pnorm{x_1}{} = 1$. Suppose we have chosen
$\{x_1, \dots, x_{i-1}\}$ that are orthonormal. Choose $x_i$ as
the one having the maximal norm among all $x\in S^{n-1}$ that are
orthogonal to $\{x_1, \dots, x_{i-1}\}$. Define a new ellipsoid
which is smaller in some directions and bigger in others:
\[
\EL=\{ \sum_{i=1}^n a_i x_i \;:\; \sum_{i=1}^{j-1}
\frac{a_i^2}{a^2} + \sum_{i=j}^n \frac{a_i^2}{b^2} \le 1 \}.
\]

Suppose, $\sum_{i=1}^n b_i x_i \in \EL$. Then $\sum_{i=1}^{j-1}
b_i x_i \in aB_2^n$, hence $\norm{\sum_{i=1}^{j-1} b_i x_i} \le
a$. Moreover, for each $x\in span \{x_j, \dots, x_n\} \bigcap
B_2^n$ we have $\pnorm{x}{} \le \pnorm{x_j}{}$  and since
$\sum_{i=j}^{n} b_i x_i \in bB_2^n$,  $\norm{\sum_{i=j}^{n} b_i
x_i} \le \norm{x_j} b$. Thus,
\[
\norm{\sum_{i=1}^n b_i x_i} \le \norm{\sum_{i=1}^{j-1} b_i x_i} +
\norm{\sum_{i=j}^n b_i x_i} \le a + \norm{x_j} \cdot b.
\]

The relation between the volumes of $\EL$ and $B_2^n$ is $Vol(\EL)
= a^{j-1}b^{n-j+1}Vol(B_2^n)$. If $a + \norm{x_j} \cdot b \le 1$,
then $\EL \subseteq K$. Using the fact that $B_2^n$ is the
ellipsoid of the maximal volume inscribed in $K$ we conclude that
\[
\forall a,b,j \text{ s.t. } a + \norm{x_j} \cdot b = 1, \;\;\;
a^{j-1}b^{n-j+1} \le 1.
\]

Substituting $b=\frac{1-a}{\norm{x_j}}$ and $a = \frac{j-1}{n}$ it
follows that for every $j \ge 2$
\[
\norm{x_j} \ge a^{\frac{j-1}{n-j+1}}(1-a) = \bigg( \frac{j-1}{n}
\bigg)^{\frac{j-1}{n-j+1}} \bigg( 1 - \frac{j-1}{n} \bigg) \ge
e^{-1} \bigg( 1 - \frac{j-1}{n} \bigg).
\]
\end{proof}

We are now ready to prove Theorem \ref{thm_dvoretzky_2} and consequently also Theorem \ref{thm_dvoretzky_1}.

As we have indicated, using Theorem \ref{thm_pre_dvoretzky}, and
assuming as we may that $B_2^n$ is the ellipsoid of maximal volume
inscribed in $K=B_{\|\cdot\|}$, it is enough to prove that
\begin{equation}\label{eq_expectation}
E = \int_{\SN} \pnorm{x} dx \ge c \sqrt{\frac{\log n}{n}},
\end{equation}
for some absolute constant $c>0$.

This will prove Theorems \ref{thm_dvoretzky_1} and \ref{thm_dvoretzky_2} with the bound $k\ge
c\frac{\varepsilon^2}{\log \frac{1}{\varepsilon}}\log n$.

We now turn to prove inequality~\ref{eq_expectation}.
According to the Dvoretzky-Rogers lemma~\ref{lem:dvor_rog} there
are orthonormal vectors $x_1, \dots, x_n$ such that for all $1
\le i \le \halfn \;\; \ \|x_i\| \ge 1/2e$.
\begin{align*}
\int_{\SN} \pnorm{x}{} d\mu(x) & = \int_{\SN} \pnorm{\sum_{i=1}^n
a_i x_i}{} d\mu(a) = \\ &= \int_{\SN} \frac{1}{2}
(\pnorm{\sum_{i=1}^{n-1} a_i x_i + a_n x_n}{} +
\pnorm{\sum_{i=1}^{n-1} a_i x_i - a_n x_n}{}) d\mu(a) \ge \\ & \ge
\int_{\SN} \max \{\pnorm{\sum_{i=1}^{n-1} a_i x_i}{}, \pnorm{a_n
x_n}{}\} d\mu(a) \ge \\ &\ge \int_{\SN} \max
\{\pnorm{\sum_{i=1}^{n-2} a_i x_i}{}, \pnorm{a_{n-1} x_{n-1}}{},
\pnorm{a_n x_n}{}\} d\mu(a) \ge \dots \ge \\ &\ge \int_{\SN}
\max_{1 \le i \le n} \pnorm{a_i x_i}{} d\mu(a) \ge \frac{1}{2e}
\int_{\SN} \max_{1 \le i \le \halfn} \modulus{a_i} d\mu(a)
\end{align*}

To Evaluate the last integral we notice that because of the
invariance of the canonical Gaussian distribution in $\RR^n$ under
orthogonal transformation and (again!) the uniqueness of the Haar
measure on $S^{n-1}$, The vector $(\sum
g_i^2)^{-1/2}(g_1,g_2,\dots,g_n)$ is distributed $\mu$. Here
$g_1,g_2,\dots,g_n$ are i.i.d. $N(0,1)$ variables. Thus

\begin{equation}\label{eq_integral}
\int_{\SN} \max_{1 \le i \le \halfn} \modulus{a_i} d\mu(a) = \EE
\frac{\max_{1 \le i \le \halfn} \modulus{g_i}}{(\sum_{i=1}^n
g_i^2)^{1/2}} = \frac{\EE \max_{1 \le i \le \halfn}
\modulus{g_i}}{\EE (\sum_{i=1}^n g_i^2)^{1/2}}
\end{equation}
(The last equation follows from the fact that the random vector $(\sum
g_i^2)^{-1/2}(g_1,g_2,\dots,g_n)$ and the random variable $(\sum
g_i^2)^{1/2}$ are independent.)

To evaluate the denominator from above note that by Jensen's
inequality:
\[
\EE (\sum_{i=1}^n g_i^2)^{1/2} \le (\EE \sum_{i=1}^n g_i^2)^{1/2}
= \sqrt{n}.
\]

The numerator is known to be of order $\sqrt{\log n}$ (estimate
the tail behavior of $\max_{1 \le i \le \halfn} \modulus{g_i}$.)

This  gives the required estimate and concludes the proof of
Theorems \ref{thm_dvoretzky_1},\ref{thm_dvoretzky_2}. \qed

As another application of Theorem \ref{thm_pre_dvoretzky} we'll
estimate the almost Euclidean sections of the $\ell_p^n$ balls
$B_p^n=\{x\in\RR^n; \|x\|_p=(\sum_{i=1}^n|x_i|^p)^{1/p}\le 1\}$.

Using the connection between the Gaussian distribution and $\mu$
we can write
\[
E_p=\int_{S^{n-1}}\|x\|_p d\mu=\EE \frac{(\sum |g_i|^p)^{1/p}}{(\sum
g_i^2)^{1/2}} = \frac{\EE (\sum |g_i|^p)^{1/p}}{\EE (\sum
g_i^2)^{1/2}}.
\]
To bound the last quantity from below we will use the following
inequality:
\[
\sqrt{2/\pi} \cdot n^{1/r} = (\sum (\EE |g_i|)^r)^{1/r}
\le \EE (\sum |g_i|^r)^{1/r} \le (\EE \sum |g_i|^r)^{1/r} = c_r \cdot
n^{1/r}
\]
Hence:
\[
E_p  \ge c_p \cdot n^{\frac{1}{p} - \frac{1}{2}}.
\]

For $p > 2$ we have $\pnorm{x}{p} \le \pnorm{x}{2}$. For $1 \le p
< 2$ we have $\pnorm{x}{p} \le n^{\frac{1}{p} - \frac{1}{2}} \cdot
\pnorm{x}{2}$. It now follows from Theorem \ref{thm_pre_dvoretzky}
that the dimension of the largest $\varepsilon$ Euclidean section of
the $\ell_p^n$ ball is

\[
k \ge \left\{%
\begin{array}{ll}
    c_p(\varepsilon) n^{\frac{2}{p}}, & 2 < p < \infty \\
    c(\varepsilon)  n, & 1 \le p < 2 .\\
\end{array}%
\right.
\]

\section { Lecture 3}
In this section we'll mostly be concerned with the question of how
good the estimates we got are. We begin with the last result of
the last section concerning the dimension of almost euclidean
sections of the $\ell_p^n$ balls.

Clearly, for $1\le p<2$ the dependence of $k$ on $n$ is best
possible. The following proposition of Bennett, Dor, Goodman,
Johnson and Newman \cite{bdgjn} shows that this is the case also for
$2<p<\infty$.

\begin{prop}
Let $2 < p < \infty$ and suppose  that $\ell_2^k$ $C$-embeds into
$\ell_p^n$, meaning that there exists a linear operator $T : \RK
\to \RN$ such that
\[
\pnorm{x}{2} \le \pnorm{Tx}{p} \le C\pnorm{x}{2},
\]
then $k \le c(p,C) n^{2/p}$.
\end{prop}

\begin{proof}
Let $T : \RK \to \RN$, $T = {(a_{ij})_{i=1}^n}_{j=1}^k$ be the
linear operator from the statement of the claim. Then for every
$x\in \RK$:
\begin{equation}\label{eq_basic}
(\sumk{j} x_j^2)^{1/2} \le (\sumn{i} \modulus{\sumk{j} a_{ij}
x_j}^p)^{1/p} \le C (\sumk{j} x_j^2)^{1/2}.
\end{equation}
In particular, for every $1 \le l \le n$, substituting instead of
$x$ the $l$-th row of $T$ we get:
\[
(\sumk{j} a_{lj}^2)^{p} \le \sumn{i} \modulus{\sumk{j} a_{ij}
a_{lj}}^p \le C^p (\sumk{j} a_{lj}^2)^{p/2}.
\]
Hence, for every $1 \le l \le n$:
\[
(\sumk{j} a_{lj}^2)^{p/2} \le C^p.
\]

Let $g_1, \dots, g_k$ be independent standard normal random
variables. Then using the fact that $\sumk{j} g_i a_j$ has the
same distribution as $(\sumk{j} a_j^2)^{1/2}g_1$ and the left hand
side of the inequality (\ref{eq_basic}) we have
\begin{align*}
\EE (\sumk{j} g_j^2)^{p/2} & \le \EE (\sumn{i} \modulus{\sumk{j}
g_j a_{ij}}^p) = \sumn{i} \EE (\modulus{g_1}^p (\sumk{j}
a_{ij}^2)^{p/2}) \le C^p \EE \modulus{g_1}^p n.
\end{align*}

On the other hand we can evaluate $\EE (\sumk{j} g_j^2)^{p/2}$
from below using the convexity of the exponent function for $p/2 >
1$:
\[
\EE (\sumk{j} g_j^2)^{p/2} \ge (\EE \sumk{j} g_j^2)^{p/2} =
k^{p/2}.
\]

Combining the last two inequalities we get an upper bound for $k$:
\[
k \le C^2 (\EE \modulus{g_1}^p)^{2/p} n^{2/p}.
\]
\end{proof}

\noindent{\bf Remarks:}
\begin{enumerate} \item There exist absolute constants $0 < \alpha \le A <
\infty$ such that $\alpha \sqrt{p} \le (\EE \modulus{g_1}^p)^{1/p}
\le A \sqrt{p}$. Hence the estimate we get for $c(p,C)$ is
$c(p,C)\le ApC^2$. In particular, for $p = \log n$, we have
\[
k \le AC^2 \log n
\]
for an absolute $A$.
 $\ell_{\log n}^n$ is $e$-isomorphic to
$\ell_\infty^n$. Hence, if we $C$-embed $\ell_2^k$ into
$\ell_\infty^n$, then $k \le Ac^2 \log n$, which means that the
$\log n$ bound in Theorem \ref{thm_dvoretzky_1} is sharp.

\item The exact dependence on
$\varepsilon$ in Theorem \ref{thm_dvoretzky_1} is an open question. From the proof we got an estimation $ k \ge \frac{c
\varepsilon^2}{\log(1/\varepsilon)} \log n$. We'll deal more with this
issue below.
\end{enumerate}

Although the last result doesn't directly give good results
concerning the dependence on $\varepsilon$ in Dvoretzky's theorem it
can be used to show that one can't expect any better beahiour on
$\varepsilon$ than $\varepsilon^2$ in Milman's theorem \ref{thm_pre_dvoretzky}. This was
observed by Tadek Figiel and didn't appear in print before. We
thank Figiel for permitting us to include it here.

\begin{claim}[Figiel]\label{claim:figiel} For any $0<\e<1$ and $n$ large enough
($n>\e^{-4}$ will do), there is a 1-symmetric norm, $\|\cdot\|$,
on $\R^n$ which is 2-equivalent to the $\ell_2$ norm and such that
if $V$ is a subspace of $\R^n$ on which the $\|\cdot\|$ and
$\|\cdot\|_2$ are $(1+\e)$-equivalent then ${\rm dim} V\le C\e^2n$
($C$ is an absolute constant).
\end{claim}

\begin{proof} Given $\e$ and $n>\e^{-4}$ (say) let $2<p<4$ be such
that $n^{\frac{1}{p}-\frac{1}{2}}=2\e$. Put
\[
\|x\|=\|x\|_2+\|x\|_p \] on $\R^n$. Assume that for some $A$ and
all $x\in V$,
\[
A\|x\|_2\le\|x\|\le (1+\e)A\|x\|_2.
\]
Clearly,
$1+\frac{\e}{2}\le\frac{1+n^{\frac{1}{p}-\frac{1}{2}}}{1+\e}\le
A\le 2$ and be get that for all $x\in V$,
\[
(A-1)\|x\|_2\le\|x\|_p\le ((1+\e)A-1)\|x\|_2=(A-1+\e A)\|x\|_2.
\]
Since $\e A\le  n^{\frac{1}{p}-\frac{1}{2}}\le 4(A-1)$, we get
that, for $B=A-1$,
\[
B\|x\|_2\le\|x\|_p\le 5B\|x\|_2.
\]
It follows from [BDGJN] that for some absolute $C$, \[ {\rm dim}
V\le Cn^{2/p}=C(n^{\frac{1}{p}-\frac{1}{2}})^2n=4C\e^2 n.\]
\end{proof}

Next we will see another relatively simple way of obtaining an
upper bound on $k$ in Dvoretzky's theorem, which, unlike the
estimate in Remark 1, tend to $0$ as $\varepsilon
\to 0$. It still leaves a big gap with the lower bound above.\\

\begin{claim}\label{cl_upper}
If $\ell_2^k$ $(1 + \varepsilon)$-embeds into $\ell_\infty^n$, then
\[
k \le \frac{C \log n}{\log(1/c\varepsilon)}\ \ ,
\]
for some absolute constants $0 < c, C < \infty$.
\end{claim}

\begin{proof}
Assume we have $(1 - \varepsilon)^{-1}$-embedding of $\ell_2^k$ into
$\ell_\infty^n$, i.e., we have a  operator $T =
{(a_{ij})_{i=1}^n}_{j=1}^k$ satisfying, for every $x \in \RK$,
\begin{equation}\label{eq_basic_2}
(1 - \varepsilon) (\sumk{j} x_j^2)^{1/2} \le \max_{1 \le i \le n}
\modulus{\sumk{j} a_{ij} x_j} \le (\sumk{j} x_j^2)^{1/2}.
\end{equation}
This means that there exist vectors $v_1, \dots, v_n \in \RK$ such
that for every $x\in \RK$:
\begin{equation}\label{eq_basic_3}
(1 - \varepsilon) \pnorm{x}{2} \le \max_{1 \le i \le n} <v_i, x> \le
\pnorm{x}{2}.
\end{equation}
In particular, $\pnorm{v_i}{2} \le 1$ for every $1 \le i \le n$.\\

Suppose $x\in \SK$, then the left hand side of~\ref{eq_basic_3}
states that there exists an $1 \le i \le n$ such that $<v_i, x>
\ge (1 - \varepsilon)$, hence:
\[
\pnorm{x - v_i}{2}^2 = \pnorm{x}{2}^2 + \pnorm{v_i}{2}^2 - 2 <v_i,
x> \le 2 - 2 (1 - \varepsilon) = 2\varepsilon.
\]

Thus, the vectors $v_1, \dots, v_n$ form a $\sqrt{2\varepsilon}$-net
on the $\SK$, which means that $n$ is much larger (exponentially)
then $k$.\\
Indeed, we have
\begin{align*}
& \bigcup_{i=1}^n B(v_i, 2 \sqrt{2\varepsilon}) \supseteq \BK
\setminus (1 - \sqrt{2\varepsilon}) \BK \\ \Rightarrow \;\;\; & n Vol
B(0,2 \sqrt{2\varepsilon}) \ge Vol B(0,1) - Vol B(0,1 -
\sqrt{2\varepsilon}) \\ \Rightarrow \;\;\; & n(2 \sqrt{2\varepsilon})^k
\ge 1 - (1 - \sqrt{2\varepsilon})^k \ge \sqrt{2\varepsilon} k(1 -
\sqrt{2\varepsilon})^{k-1}.
\end{align*}

This gives for $\varepsilon < \frac{1}{32}$ and $k \ge 12$
\[
n \ge \frac{k}{2} (\frac{1}{4 \sqrt{2\varepsilon}})^{k-1} \ge
(\frac{1}{4 \sqrt{2\varepsilon}})^{k/2},
\]
or
\[
k \le \frac{4\log n}{\log \frac{1}{32 \varepsilon}}.
\]
\end{proof}

This shows that the $c(\e)$ in the statement of Theorem \ref{thm_dvoretzky_1} can't be larger than $\frac{C}{\log(1/c\varepsilon)}$.

Our last objective in this survey is to improve
somewhat the lower estimate on $c(\e)$ in the version of
Dvoretzky's theorem we proved. For that we'll need the inverse to
Claim \ref{cl_upper}.

\begin{claim}\label{cl_inverse}
$\ell_2^k$ $(1 + \varepsilon)$-embeds into $\ell_\infty^n$ for
\[
k = \frac{c \log n}{\log(1/c\varepsilon)}\ \ ,
\]
for some absolute constants $0 < c, C < \infty$.
\end{claim}

The proof is very simple and we only state the embedding. Use
Lemma \ref{lem_first} to find an $\e$-net $\{x_i\}_{i=1}^n$ on
$s^{k-1}$ where $k$ and $n$ are related as in the statement of the
claim. The embedding of $\ell_2^k$ into $\ell_\infty^n$ is given
by $x\to \{\langle x,x_i\rangle\}_{i=1}^n$.

\section {Lecture 4}
In this last section we'll prove a somewhat improved version of
Dvoretzky's theorem, replacing the $\e^2$ dependence by $\e$
(except for a $\log$ factor).

\begin{thm}\label{thm_mine} There is a constant $c>0$ such
that for all $n\in \NN$ and all $\e>0$, every $n$-dimensional
normed space $\ell_2^k$ $(1+\varepsilon)$-embeds in
$(\RN,\pnorm{\cdot}{})$ for some $k\ge \frac
{c\e}{(\log\frac{1}{\e})^2}{\log n}$.
\end{thm}

The idea of the proof is the following: We start as in the proof
of Milman's theorem \ref{thm_pre_dvoretzky}, assuming $S^{n-1}$ is
the ellipsoid of maximal volume inscribed in the unit ball of
$B_{\pnorm{\cdot}{}}$. If $E$ is large enough (so that
$\e^2E^2n\ge \frac {\e}{(\log\frac{1}{\e})^2}{\log n}$) we get the
result from Milman's theorem. If not, we'll show that the space
actually contains a relatively high dimensional $\ell_\infty^m$
and then use Claim \ref{cl_inverse} to get an estimate on the
dimension of the embedded $\ell_2^k$.

The main proposition is the following one which improves the main proposition of \cite{sc2}:

\begin{prop}\label{prop:main} Let $(X,\|\cdot\|)$ be a normed space and let $x_1,\dots,x_{n}$ be
a  sequence in $X$ satisfying $\|x_i\|\ge1/10$ for all $i$ and
\begin{equation}\label{eq:log}
\EE\Big(\|\sum_{i=1}^{n} g_i x_i\|\Big)\le L\sqrt{\log n}.
\end{equation}
Then, there is a subspace of $X$ of dimension $k\ge
\frac{n^{1/4}}{CL}$ which is $CL$-isomorphic to $\ell_\infty^k$.
$C$ is a universal constant.
\end{prop}

\noindent Let us assume the proposition and continue with the

\begin{proof}[Proof of Theorem \ref{thm_mine}] We start as
in the proof of Theorem \ref{thm_dvoretzky_1}, assuming $B_2^n$ is the
ellipsoid of maximal volume inscribed in the unit ball of
$(\R^n,\|\cdot\|)$. As we already said we may assume $\e^2E^2n\le
\frac {\e}{(\log\frac{1}{\e})^2}{\log n}$ or $E\sqrt{n}\le
\frac{\sqrt{\log n}}{{\sqrt{\e}}\log\frac{1}{\e}}$. Let
$x_1,\dots, x_n$ be the orthonormal basis given by the
Dvoretzky--Rogers Lemma, so that in particular $\|x_i\|\ge 1/10$
for $i=1,\dots,n/2$. It follows from the triangle inequality for the first inequality and from the relation between the distribution of a canonical Gaussian vector and the Haar measure on the sphere that
\[
\EE\Big(\|\sum_{i=1}^{n/2} g_i x_i\|\Big)\le \EE\Big(\|\sum_{i=1}^{n}
g_i x_i\|\Big)\le CE\sqrt{n}
\]
So,
\[
\EE\Big(\|\sum_{i=1}^{n/2} g_i x_i\|\Big)\le \frac{\sqrt{\log
n}}{{\sqrt{\e}}\log\frac{1}{\e}}.
\]
and by Proposition \ref{prop:main} there is a subspace of
$(\R^n,\|\cdot\|)$ of dimension $k\ge \frac{n^{1/4}}{CL}$ which is
$CL$-isomorphic to $\ell_\infty^k$ where
$L=\frac{1}{{\sqrt{\e}}\log\frac{1}{\e}}$. It now follows from an
iteration result of James (see Lemma \ref{lem_james} below and
Corollary \ref{cor_james} following it) that for any $0<\e<1$
there is a subspace of $(\R^n,\|\cdot\|)$ of dimension $k\ge
cn^{\frac{c\e}{\log L}}$ which is $1+\e$ - isomorphic to
$\ell_\infty^k$. $c>0$ is a universal constant. We now use Claim
\ref{cl_inverse} to conclude that $\ell_2^k$ embeds in our space
for some $k \ge \frac{c \log (cn^{\frac{c\e}{\log
L}})}{\log(1/c\varepsilon)}=\frac {c^\prime\e\log
n}{(\log(1/c\varepsilon))^2}$.
\end{proof}

The following simple Lemma is due to R. C. James

\begin{lem}\label{lem_james} let $x_1,\dots,x_m$ be vectors
in some normed space $X$ such that $\|x_i\|\ge 1$ for all $i$ and
\[
\|\sum_{i=1}^ma_ix_i\|\le L\max_{1\le i\le m}|a_i|
\]
for all sequences of coefficients $a_1,\dots,a_m\in\R$. Then $X$
contains a sequence $y_1,\dots,y_{\lfloor\sqrt{m}\rfloor}$
satisfying $\|y_i\|\ge 1$ for all $i$ and
\[
\|\sum_{i=1}^{\lfloor\sqrt{m}\rfloor}a_iy_i\|\le
\sqrt{L}\max_{1\le i\le {\lfloor\sqrt{m}\rfloor}}|a_i|
\]
for all sequences of coefficients
$a_1,\dots,a_{\lfloor\sqrt{m}\rfloor}\in\R$.
\end{lem}

\begin{proof} Let $\sigma_j$, $j=1,\dots,{\lfloor\sqrt{m}\rfloor}$
be disjoint subsets of $\{1,\dots,m\}$ each of cardinality $\lfloor\sqrt{m}\rfloor$. If for some $j$
\[
\|\sum_{i\in\sigma_j}a_ix_i\|\le \sqrt{L}\max_{i\in\sigma_j}|a_i|
\]
for all sequences of coefficients, we are done. Otherwise, for
each $j$ we can find a vector $y_j=\sum_{i\in\sigma_j}a_ix_i$ such
that $\|y_j\|=1$ and $\sqrt{L}\max_{i\in\sigma_j}|a_i|<1$. But
then,
\[
\|\sum_{j=1}^{\lfloor\sqrt{m}\rfloor}b_jy_j\|\le L\max_{j,\
i\in\sigma_j}|b_ja_i|\le
L\max_j|b_j|\sqrt{L^{-1}}=\sqrt{L}\max_j|b_j|.
\]
\end{proof}

\begin{cor}\label{cor_james} If $\ell_\infty^m$ $L$-embeds into a
normed space $X$, then for all $0<\e<1$, $\ell_\infty^k$
$\frac{1+\e}{1-\e}$-embeds into $X$ for $k\sim m^{\e/\log L}$.
\end{cor}

\begin{proof}
By iterating the Lemma (pretending for the sake of simplicity of notation that $m^{2^{-s}}$ is an integer for all the relevant $s$-s), for all positive integer $t$ there is
a sequence of length $k=m^{2^{-t}}$ of norm one vectors
$x_1,\dots,x_k$ in $X$ satisfying
\[
\|\sum_{i=1}^ka_ix_i\|\le L^{2^{-t}}\max|a_i|
\]
for all coefficients. Pick a $t$ such that $L^{2^{-t}}=1+\e$
(approximately); i.e., $2^{-t}=\frac{\log 1+\e}{\log L}\sim
\frac{\e}{\log L}$. Thus $k\sim m^{\e/\log L}$ and
\[
\|\sum_{i=1}^ka_ix_i\|\le (1+\varepsilon)\max|a_i|.
\] 
To get a
similar lower bound on $\|\sum_{i=1}^ka_ix_i\|$, assume without
loss of generality that $\max|a_i|=a_1$. Then
\[
\begin{array}{rl}
\|\sum_{i=1}^ka_ix_i\|=&\|2a_1x_1-(a_1x_1-\sum_{i=2}^ka_ix_i)\|\ge
2a_1-\|a_1x_1-\sum_{i=2}^ka_ix_i\|\\
\ge& 2a_1-(1+\e)a_1=(1-\e)\max|a_i|.
\end{array}
\]
\end{proof}

We are left with the task of proving Proposition \ref{prop:main}. We begin with

\begin{claim}\label{cl1}
Let $x_1,\dots,x_n$ be normalized vectors in a normed space. Then
for all real $a_1,\dots,a_n$,
\[
{\rm Prob}_{\e_i=\pm 1}(\|\sum_{i=1}^n \e_ia_ix_i\|<\max_{1\le i\le n}|a_i|)\le 1/2.
\]
\end{claim}

\begin{proof} Assume as we may $a_1=\max_{1\le i\le n}|a_i|$. If
$\|a_1x_1+\sum_{i=2}^n \e_ia_ix_i\|<a_1$ then
\[
\|a_1x_1-\sum_{i=2}^n \e_ia_ix_i\|\ge 2a_1-\|a_1x_1+\sum_{i=2}^n
\e_ia_ix_i\|>a_1
\] and thus
\[
P(\|\sum_{i=1}^n \e_ia_ix_i\|>a_1)\ge P(\|\sum_{i=1}^n
\e_ia_ix_i\|<a_1).
\]
So,
\begin{equation*}
\begin{array}{rl}
1\ge&P(\|\sum_{i=1}^n \e_ia_ix_i\|\not=\max|a_i|)\cr
=&P(\|\sum_{i=1}^n \e_ia_ix_i\|<a_1)+P(\|\sum_{i=1}^n
\e_ia_ix_i\|>a_1)\cr \ge& 2P(\|\sum_{i=1}^n \e_ia_ix_i\|<a_1).
\end{array}
\end{equation*}
\end{proof}

\begin{remark} If $x_1=x_2$, $a_1=a_2=1$ and $a_3=\dots=a_n=0$ then the $1/2$ in the statement of Claim \ref{cl1} cannot be
replaced by any smaller constant.
\end{remark}

\begin{prop}\label{prop:smallball}
Let $x_1,\dots,x_n$ be vectors in a normed space with $\|x_i\|\ge
1/10$ for all $i$ and let $g_1,\dots,g_n$ be a sequence of independent standard Gaussian variables. Then, for $n$ large enough,
\[
P(\|\sum_{i=1}^n g_ix_i\|<\frac{\sqrt{\log n}}{100})\le 2/3.
\]
\end{prop}

\begin{proof}
Note first that it follows from Claim \ref{cl1} that
\begin{equation}\label{eq:gau}
P(\|\sum_{i=1}^n g_ix_i\|<\max_{1\le i\le
n}|g_i|\|x_i\|)\le \frac12.
\end{equation}
This is easily seen by noticing that $(g_1\dots,g_n)$ is distributed identically to 
$(\varepsilon_1|g_1|\dots,\varepsilon_n|g_n|)$ where $\varepsilon_1\dots,\varepsilon_n$ are independent random signs independent of the $g_i$-s. Now compute 
\[
P(\|\sum_{i=1}^n \varepsilon_i|g_i|x_i\|<\max_{1\le i\le
n}|g_i|\|x_i\|)
\]
by first conditioning on the $g_i$-s. We use (\ref{eq:gau}) in the following sequence of inequalities.

\begin{equation*}
\begin{array}{rl}
&P(\|\sum_{i=1}^n g_ix_i\|<\frac{\sqrt{\log n}}{100})\cr
&\phantom{aaa}\le P(\|\sum_{i=1}^n g_ix_i\|<\frac{\sqrt{\log
n}}{100} \ \&\ \frac{\sqrt{\log n}}{100}<\max_{1\le i\le
n}|g_i|\|x_i\|)\cr &\phantom{aaaaaa}+P(\max_{1\le i\le
n}|g_i|\|x_i\|\le\frac{\sqrt{\log n}}{100})\cr &\phantom{aaa}\le
P(\|\sum_{i=1}^n g_ix_i\|<\max_{1\le i\le
n}|g_i|\|x_i\|)+P(\max_{1\le i\le n}|g_i|\le\frac{\sqrt{\log
n}}{10})\cr &\phantom{aaa}\le \frac12+(1-e^{-c\log n})^n\ \quad
\quad \quad \quad \quad \quad \mbox{for $n$ large enough}\cr
&\phantom{aaa}\le \frac12+e^{-n^{1-c}}\le \frac23.
\end{array}
\end{equation*}
\end{proof}
 In the proof of Proposition \ref{prop:main} we shall use a theorem of Alon and Milman \cite{am} (see \cite{ta} for a simpler proof) which have a very similar statement: Gaussians are replaced by random signs and $\sqrt {\log n}$ by a constant.

 \begin{thm}\label{thm:am}(Alon and Milman) Let $(X,\|\cdot\|)$ be a normed space and let $x_1,\dots,x_{n}$ be
a  sequence in $X$ satisfying $\|x_i\|\ge1$ for all $i$ and
\begin{equation}\label{eq:log}
\EE_{\e_i=\pm 1}\Big(\|\sum_{i=1}^{n} \e_i x_i\|\Big)\le L.
\end{equation}
Then, there is a subspace of $X$ of dimension $k\ge
\frac{n^{1/2}}{CL}$ which is $CL$-isomorphic to $\ell_\infty^k$.
$C$ is a universal constant.
\end{thm}

{\em Proof of Proposition \ref{prop:main}.} Let $\sigma_1,\dots,\sigma_{\lfloor
\sqrt{n}\rfloor}\subset\{1,\dots,n\}$ be disjoint
 with $|\sigma_j|=\lfloor
\sqrt{n}\rfloor$ for all $j$. We'll show that there is a subset
$J\subset\{1,\dots,\lfloor \sqrt{n}\rfloor\}$ of cardinality at
least $\frac{\sqrt{n}}{4}$
 and there are
$\{y_j\}_{j\in J}$ with $y_j$ supported on $\sigma_j$ such that
$\|y_j\|=1$ for all $j\in J$ and
\[
\EE_{\e_i=\pm 1}\Big(\|\sum_{j\in J} \e_j y_j\|\Big)\le 80L.
\]
We then apply the theorem above.

To show this notice that the events $\|\sum_{i\in \sigma_j} g_ix_i\|<\frac{\sqrt{\log n}}{200}$, $j=1,\dots, \lfloor
\sqrt{n}\rfloor$, are independent and by Proposition \ref{prop:smallball} have probability at most $2/3$ each. So with probability  at least $1/2$ there is a subset  $J\subset \{1,\dots, \lfloor
\sqrt{n}\rfloor\}$ with $|J|\ge \frac {\lfloor
\sqrt{n}\rfloor}{4}$ such that $\|\sum_{i\in \sigma_j} g_i x_i\|>
\frac{1}{200}\sqrt{\log n}$ for all $j\in J$. Denote the event that
such a $J$ exists by $A$. Let $\{r_j\}_{j=1}^{\lfloor
\sqrt{n}\rfloor}$ be a sequence of independent signs independent of the
original Gaussian sequence. We get that
\[
\begin{array}{rl}
L\sqrt{\log n}\ge& \EE_g\Big(\|\sum_{j=1}^{\lfloor
\sqrt{n}\rfloor}\sum_{i\in\sigma_j}g_ix_i\|\Big)=
\EE_r\EE_g\Big(\|\sum_{j=1}^{\lfloor
\sqrt{n}\rfloor}r_j\sum_{i\in\sigma_j}g_ix_i\|\Big)\\
\ge&\EE_r\EE_g\Big(\|\sum_{j=1}^{\lfloor
\sqrt{n}\rfloor}r_j\sum_{i\in\sigma_j}g_ix_i\|{\bf 1}_A\Big)\\
\ge&
\frac12 \EE_g\Big(\Big(E_r\|\sum_{j=1}^{\lfloor
\sqrt{n}\rfloor}r_j\sum_{i\in\sigma_j}g_ie_i\|\Big)\Big/ A\Big).
\end{array}
\]
It follows that for some $\omega\in A$,  there exists a $J\subset
\{1,\dots, {\lfloor \sqrt{n}\rfloor}\}$ with $|J|\ge
\frac{{\lfloor \sqrt{n}\rfloor}}{4}$ such that putting $\bar
y_j=\sum_{i\in\sigma_j}g_i(\omega)x_i$, one has $\|\bar
y_j\|>\frac{1}{200}\sqrt{\log n}$ for all $j\in J$ and
\[
\EE_r\Big(\|\sum_{j\in J}r_j\bar y_j\|\Big)\le 2L\sqrt{\log n}.
\]
Take $y_j=\bar y_j/\|\bar y_j\|$.
\qed

\bigskip

In the list of references below we included also some books and expository papers not directly referred to in the text above.

\noindent Gideon Schechtman\newline Department of
Mathematics\newline Weizmann Institute of Science\newline Rehovot,
Israel\newline E-mail: gideon.schechtman@weizmann.ac.il


\begin{thebibliography}{99}

\bibitem[AM]{am}
N. Alon and V.D. Milman, Embedding of $l\sp{k}\sb{\infty }$ in
finite-dimensional Banach spaces. Israel J. Math. 45 (1983), no.
4, 265--280.

\bibitem[BDGJN]{bdgjn} G. Bennett, L.E. Dor, V. Goodman, W.B. Johnson, C. Newman, On uncomplemented subspaces of $L\sb{p}$, $1<p<2$,
 Israel J. Math.  26  (1977),  no. 2, 178--187.

\bibitem[Dv]{dv} A. Dvoretzky, Some results on convex bodies
and Banach spaces,
 1961  Proc. Internat. Sympos. Linear Spaces (Jerusalem, 1960)
 pp. 123--160 Jerusalem Academic Press, Jerusalem; Pergamon, Oxford.

\bibitem[DR]{dr} A. Dvoretzky, C.A. Rogers, Absolute and
unconditional convergence in normed linear spaces,
 Proc. Nat. Acad. Sci. U. S. A.  36,  (1950), 192--197.

\bibitem[Fi]{fi} T. Figiel, A short proof of Dvoretzky's theorem
on almost spherical sections of
 convex bodies,
 Compositio Math.  33  (1976),  no. 3, 297--301.

\bibitem[GM]{gm} A.A. Giannopoulos, V.D. Milman,
Euclidean structure in finite dimensional normed spaces,
 Handbook of the geometry of Banach spaces, Vol. I,
 707--779, North-Holland, Amsterdam,  2001.
 
 \bibitem[Go]{go} Y. Gordon, Some inequalities for Gaussian processes and applications. Israel J. Math. 50 (1985), no. 4, 265--289.

\bibitem[Gr]{gr} A. Grothendieck, Sur certaines classes de
suites dans les espaces de Banach et le
 th\'eor\`eme de Dvoretzky-Rogers,
(French)  Bol. Soc. Mat. São Paulo  8  1953 81--110 (1956).

\bibitem[JS]{js} W.B. Johnson and G. Schechtman, Finite dimensional subspaces of
$L_p$\/, Handbook of the geometry of Banach spaces, Vol. I, 837--870, North-Holland, Amsterdam, 2001.

\bibitem[Mi]{mi} V.D. Milman, A new proof of A. Dvoretzky's
theorem on cross-sections of convex
 bodies,
(Russian)  Funkcional. Anal. i Priložen.  5  (1971),  no. 4, 28--37.

\bibitem[MS]{ms}
V.M. Milman and G. Schechtman, Asyptotic theory of
finite-dimensional normed spaces\/, Lecture Notes in Mathematics,
1200, Springer-Verlag, Berlin, 1986.

\bibitem[Pi]{Pis89} G. Pisier, The volumes of convex bodies and Banach space geometry\/,
Cambridge University Press, Cambridge 1989.

\bibitem[Sc1]{sc0} G. Schechtman,A remark concerning the dependence on $\varepsilon$ in Dvoretzky's theorem. Geometric aspects of functional analysis (1987–88), 274--277, Lecture Notes in Math., 1376, Springer, Berlin, 1989.

\bibitem[Sc2]{sc1} G. Schechtman, Concentration, results and
applications\/, Handbook of the geometry of Banach spaces, Vol. 2, 1603--1634, North-Holland, Amsterdam, 2003.

\bibitem[Sc3]{sc2} G. Schechtman, Two observations
regarding embedding subsets of Euclidean spaces in
 normed spaces,
 Adv. Math.  200  (2006),  no. 1, 125--135.

\bibitem[Schn]{Schn} R. Schneider, Convex bodies: the Brunn-Minkowski
theory\/, Encyclopedia of Mathematics and its Applications, 44.
Cambridge University Press, Cambridge, 1993.

\bibitem[Sz]{sz} A. Szankowski, On Dvoretzky's theorem on
almost spherical sections of convex
 bodies,
 Israel J. Math.  17  (1974), 325--338.

\bibitem[Ta]{ta}
M. Talagrand, Embedding of $l\sp \infty\sb k$ and a theorem of
Alon and Milman. Geometric aspects of functional analysis
(Israel, 1992--1994), 289--293, Oper. Theory Adv. Appl., 77,
Birkhauser, Basel, 1995.





\end{thebibliography}
\end{document}